\documentclass[12pt]{amsart}
\usepackage{latexsym, amsmath, amssymb, amsthm, mathrsfs}
\usepackage[alwaysadjust]{paralist}
\usepackage{tikz}
\usetikzlibrary{arrows,shapes,positioning,calc,patterns}
\usepackage{caption}
\usepackage{subcaption}
\usepackage[colorlinks=true,linkcolor=blue,urlcolor=blue, citecolor=blue]%
  {hyperref}
\usepackage[shortalphabetic]{amsrefs}
\usepackage[T1]{fontenc}
\usepackage{mathptmx}
\usepackage{microtype}

\usepackage[centering, includeheadfoot, hmargin=1in, vmargin=1in,
  headheight=30.4pt]{geometry}
\newtheorem{lemma}{Lemma}[section]
\newtheorem{theorem}[lemma]{Theorem}
\newtheorem{corollary}[lemma]{Corollary}
\newtheorem{proposition}[lemma]{Proposition}
\theoremstyle{definition}
\newtheorem{remark}[lemma]{Remark}
\newtheorem{example}[lemma]{Example}
\newtheorem{question}[lemma]{Question}
\newcommand{\define}[1]{{\bfseries\itshape #1}}
\newcommand{\norm}[1]{\ensuremath{\| #1 \|}}
\newcommand{\xqed}[1]{%
  \leavevmode\unskip\penalty9999 \hbox{}\nobreak\hfill
  \quad\hbox{\ensuremath{#1}}}
\renewcommand{\theequation}%
{\arabic{section}.\arabic{lemma}.\arabic{equation}}
\newcommand{\relphantom}[1]{\mathrel{\phantom{#1}}}
\newcommand{\CC}{\ensuremath{\mathbb{C}}} 
\newcommand{\NN}{\ensuremath{\mathbb{N}}} 
\newcommand{\PP}{\ensuremath{\mathbb{P}}} 
 
\newcommand{\RR}{\ensuremath{\mathbb{R}}} 
\newcommand{\ZZ}{\ensuremath{\mathbb{Z}}} 
\newcommand{\sF}{\ensuremath{\mathscr{F}}} 
\newcommand{\sI}{\ensuremath{\kern -1pt \mathscr{I}\kern -2pt}} 
\newcommand{\sJ}{\ensuremath{\kern -2pt \mathscr{J}\kern -2pt}} 
\newcommand{\sO}{\ensuremath{\mathscr{O}}} 
\newcommand{\sOX}{\ensuremath{\mathscr{O}^{}_{\! X}}} 
\newcommand{\sOXp}{\ensuremath{\mathscr{O}^{}_{\! X'}}} 
\newcommand{\KX}{\ensuremath{K^{}_{X}}} 
\newcommand{\KXp}{\ensuremath{K^{}_{X'}}} 
\renewcommand{\geq}{\geqslant}
\renewcommand{\leq}{\leqslant}
\DeclareMathOperator{\Bl}{Bl}
\DeclareMathOperator{\codim}{codim}
\DeclareMathOperator{\mult}{mult}
\DeclareMathOperator{\Nef}{Nef}
\DeclareMathOperator{\Pic}{Pic}
\DeclareMathOperator{\sTor}{{\mathcal{T}\!\!\textit{or}}}
\DeclareMathOperator{\Spec}{Spec}
\DeclareMathOperator{\Sym}{Sym}
\DeclareMathOperator{\Supp}{Supp}

\begin{document}

\title[Vanishing Theorems]{Vanishing Theorems and the Multigraded Regularity\\
  of Nonsingular Subvarieties}

\author[V.~Lozovanu]{Victor Lozovanu}

\author[G.G.~Smith]{Gregory G. Smith}

\address{Department of Mathematics and Statistics \\ Queen's
  University\\ Kingston \\ ON \\ K7L~3N6\\ Canada}
\email{\href{mailto:vicloz@mast.queensu.ca}{vicloz@mast.queensu.ca}, \href{mailto:ggsmith@mast.queensu.ca}{ggsmith@mast.queensu.ca}}



\begin{abstract}
  $\!$ Given scheme-theoretic equations for a nonsingular subvariety, we prove
  that the higher cohomology groups for suitable twists of the corresponding
  ideal sheaf vanish.  From this result, we obtain linear bounds on the
  multigraded Castelnuovo-Mumford regularity of a nonsingular subvariety, and
  new criteria for the embeddings by adjoint line bundles to be projectively
  normal.  A special case of our work recovers the vanishing theorem of Bertram,
  Ein, and Lazarsfeld.
\end{abstract}

\maketitle

\section{Introduction} 
\label{sec:intro}

\noindent
In higher dimensional geometry, vanishing theorems are indispensable for
uncovering the deeper relations between the geometry of a subvariety and its
defining equations.  The main result of this paper, a new vanishing theorem for
certain twists of the ideal sheaf of a nonsingular subvariety, adds to this
framework.  Indeed, our vanishing theorem leads to linear bounds on the
multigraded Castelnuovo-Mumford regularity of nonsingular subvarieties and new
conditions for the surjectivity of the multiplication maps between global
sections of adjoint line bundles.  Our techniques also yield a new
Griffiths-type vanishing theorem for vector bundles.

Let $X$ be a nonsingular complex projective variety with canonical line bundle
$\KX$.  For a nonsingular subvariety $Y \subseteq X$ of codimension $e$ with
ideal sheaf $\sI^{}_Y$, the following is our main theorem.

\begin{theorem}
  \label{thm:main}
  Let $L$ be a line bundle on $X$ and let $m$ be a nonnegative integer.  If $Y$
  is defined scheme-theoretically by the nef divisors $D^{}_1, \dotsc, D^{}_r$
  and $L \otimes \sOX\bigl( -(m+1)D_{s_1} - D_{s_2} - \dotsb - D_{s_e} \bigr)$
  is a big, nef line bundle for all subsets $\{s^{}_1, s^{}_2, \dotsc, s^{}_e \}
  \subseteq \{1, \dotsc, r\}$, then we have
  \[
  H^i \bigl( X, \sI_{Y}^{\,\, m+1} \otimes \KX \otimes L \bigr) = 0 \, , \quad
  \text{for all $i > 0$.}
  \]
\end{theorem}

Under the assumption that each $D_{\! j}$ belongs to the linear system for some
power of a single globally generated line bundle, we recover Theorem~7 in
\cite{BEL}.  In particular, this additional hypothesis induces an ordering on
the subsets $\{ s^{}_1, \dotsc s^{}_e \}$ and it is enough to consider the
unique maximal subset.  Unlike both Theorem~7 in \cite{BEL} or Theorem~1.1 in
\cite{FE}, we do not assume that there exists such an ordering.  This
distinction is surprisingly significant; we don't have to follow a single ray in
the nef cone of $X$.  At a superficial level, our statement must involve an
extra universal quantifier because we cannot select a maximal subset.  By
capitalizing on local information, Theorem~\ref{thm:semilocal} allows one to
avoid considering all subsets.  More substantially, the arguments in \cite{BEL},
\cite{Ber} and \cite{FE} all implicitly exploit an ordering, so we need a
different proof.

Although a scheme $Y$ should be studied from within its `natural' ambient space
$X$, the advantages of our main theorem are most pronounced when the nef cone of
$X$ has dimension at least two.  For example, it is tailor-made to handle the
ambient toric varieties that appear in mirror symmetry~\cite{CK}, embeddings of
Mori dream spaces \cite{HK}, and tropical compactifications \cite{Tev}.  As
Remark~\ref{rem:KawamataViehweg} explains, the main theorem is a
higher-codimensional version of the Kawamata-Viehweg vanishing theorem.
Proposition~\ref{pro:sharp} and Example~\ref{exa:product} demonstrate that it is
optimal from this perspective.

Our primary application concerns the multigraded Castelnuovo-Mumford regularity
of $Y$.  Let $L$ be a line bundle on $X$ and fix free divisors $P^{}_1, \dotsc,
P^{}_\ell$ on $X$ such that some positive linear combination is very ample.
Following \cite{MS} or \cite{HSS}, a sheaf $\sF$ is \define{$L\,$-regular} with
respect to $P^{}_1, \dotsc, P^{}_\ell$ if $H^i \bigl( X, \sF \otimes L \otimes
\sOX(- u^{}_1 \, P^{}_1 - \dotsb - u^{}_\ell \, P^{}_\ell) \bigr) = 0$ for all
$i > 0$ and all $(u^{}_1, \dotsc, u^{}_\ell) \in \NN^\ell$ satisfying $u^{}_1 +
\dotsb + u^{}_\ell = i$.  This is the standard form of Castelnuovo-Mumford
regularity when $\ell = 1$ (e.g., see Section~1.8 of \cite{PAGI}).  Since the
globally generated line bundles $\sOX(P^{}_1), \dotsc, \sOX(P^{}_\ell)$ do not
have a natural total order, one cannot choose a smallest regularity.  Bounding
the regularity of $\sF$ is, therefore, equivalent to describing a subset of
$\bigl\{ L \in \Pic(X) : \text{$\sF$ is $L$-regular} \bigr\}$.  By taking $m =
0$ in Theorem~\ref{thm:main}, we obtain the following bounds on the multigraded
regularity of $\sI^{}_{Y}$.

\begin{corollary}
  \label{cor:reg}
  Let $N$ be a nef line bundle on $X$ and let $Y \subseteq X$ be defined
  scheme-theoretically by nef divisors $D_1, \dotsc, D_r$.  Assume that each
  $D_j$ can be expressed as a positive linear combination of the $P^{}_1, \dotsc
  P^{}_\ell$.  If $L \otimes \sOX(- D_{s_1} - \dotsb - D_{s_e} - u^{}_1 \,
  P^{}_1 - \dotsb - u^{}_\ell \, P^{}_\ell)$ is a big, nef line bundle for all
  subsets $\{s^{}_1, \dotsc, s^{}_e \} \subseteq \{1, \dotsc, r\}$ and all
  $(u_1, \dotsc, u_\ell) \in \NN^\ell$ satisfying $u^{}_1 + \dotsb + u^{}_\ell =
  \dim(Y) + 1$, then the ideal sheaf $\sI^{}_Y$ is $K_X \otimes L \otimes
  N$-regular with respect to $P^{}_1, \dotsc, P^{}_\ell$.
\end{corollary}

\noindent
As an immediate consequence, we see that the multigraded regularity of $Y$ grows
at most linearly in terms of its defining equations (see
Remark~\ref{rem:linear}).  In other words, even when embedded into an ambient
space other than $\PP^d$, the `algebraic complexity' of a nonsingular variety is
essentially as small as possible.  In the special case that $X = \PP^d$ and
$\sOX(P^{}_1) = \sOX(1)$, Example~\ref{exa:reg} shows that
Corollary~\ref{cor:reg} becomes the sharp bound in Corollary~4~(i) in
\cite{BEL}.  Since Theorem~4.7 in \cite{CU} and Corollary~1.2 in \cite{FE} both
generalize Corollary~4 in \cite{BEL} by allowing $X$ and $Y$ to have mild
singularities, one certainly hopes that Corollary~\ref{cor:reg} can be
strengthened in a similar way, but our techniques do not currently accommodate
singularities.

Our second application involves adjoint line bundles, that is line bundles of
the form $\KX \otimes P$ where $P$ is a suitably positive line bundle.  Set $d
:= \dim(X)$ and let $A^{}_1, \dotsc, A^{}_d$ be very ample line bundles on $X$.
Applying Theorem~\ref{thm:semilocal} to the diagonal $\Delta(X) \subset X \times
X$ yields the following criterion. 

\begin{corollary}
  \label{cor:veryample}
  If $L^{}_1$ and $L^{}_2$ are line bundles on $X$ such that $L_j \otimes
  A^{-1}_1 \otimes \dotsb \otimes A^{-1}_d$ is big and nef for all $1 \leq j
  \leq 2$, then the multiplication map
  \[
  H^0 \bigl(X, \KX \otimes L^{}_1 \bigr) \otimes H^0 \bigl( X, \KX \otimes
  L^{}_2 \bigr) \longrightarrow H^0 \bigl( X, K^{\otimes 2}_X \otimes L^{}_1
  \otimes L^{}_2 \bigr)
  \]
  is surjective.  In particular, the adjoint bundle $\KX \otimes L^{}_1$ defines
  a projectively normal embedding of $X$ provided it is very ample.
\end{corollary}

\noindent
We recover Variant~3.2 in \cite{BEL} when $A^{}_1 = \dotsb = A^{}_d$ and
Example~\ref{exa:K3} highlights the comparative power of
Corollary~\ref{cor:veryample}.  By replacing the very ample line bundles with
globally generated line bundles that define a closed embedding into a product of
projective spaces, Proposition~\ref{pro:adjoint} provides an analogous criteria
for the multiplication map to be surjective.  Example~\ref{exa:comparative}
illustrates the utility of these alternative hypotheses.  Finally, by employing
Theorem~\ref{thm:main} with $m > 0$, Proposition~\ref{pro:Wahl} gives similar
surjectivity statements for the $m$-th Wahl map.

In contrast with the multiplier ideal methods in \cite{Ber} and \cite{FE}, the
key tool for proving Theorem~\ref{thm:main} is an asymptotic multiplier ideal.
Although the basic outline of our proof parallels \cite{Ber}, the asymptotic
variant delivers three pivotal features.  First, there is a slightly stronger
variant of Nadel vanishing involving only a big line bundle (e.g., see
Theorem~11.2.12~(ii) in \cite{PAGI}).  Second, asymptotic multiplier ideals
depend only on the numerical class of a line bundle (see Example~11.3.12 in
\cite{PAGI}) and this provides some leeway in establishing that a given ideal is
trivial.  Third, we can use asymptotic constructions such as a generalization of
Wilson's theorem (see Lemma~\ref{lem:Wilson}).  Our approach, especially
Lemma~\ref{lem:technical} and Theorem~\ref{thm:semilocal}, underscores an
intriguing connection between local properties of the divisors $D_1, \dotsc,
D_r$ and vanishing statements.  As an added benefit, these ideas also produce
the following vanishing theorem for vector bundles.

\begin{proposition}
  \label{pro:Griffiths}
  Let $E$ be a vector bundle of rank $e$ on $X$ that is a quotient of a direct
  sum of line bundles $\sOX(D^{}_1) \oplus \dotsb \oplus \sOX(D^{}_r)
  \longrightarrow E$ where each divisor $D_j$ is nef.  If $m$ is a positive
  integer and the line bundle $L \otimes \sOX\bigl( (m+1) D_{s_1} + D_{s_2} +
  \dotsb + D_{s_e} \bigr)$ is big, nef for all subsets $\{ s^{}_1, \dotsc,
  s^{}_e \} \subseteq \{1, \dotsc, r\}$, then we have $H^i \bigl( X, \KX \otimes
  \det(E) \otimes \Sym^m(E) \otimes L \bigr) = 0$ for all $i > 0$.
\end{proposition}

\noindent 
Just as with the main theorem, we recover Proposition~1.12 in \cite{BEL} under
the assumption that each $D_{\! j}$ belongs to the linear system for some power
of a single globally generated line bundle.  Assuming $E$ is a quotient of a
direct sum of nef line bundles, Proposition~\ref{pro:Griffiths} also improves on
Example~7.3.3 in \cite{PAGII}.

\section{Notation and conventions}

\begin{itemize}
\item We write $\NN$ for the set of nonnegative integers.
\item We work throughout over the complex numbers $\CC$.
\item Let $X$ be a nonsingular projective variety of dimension $d$ and let $K_X$
  be its canonical line bundle.
\item A subvariety $Y \subseteq X$ is \define{defined scheme-theoretically} by
  the divisors $D_1, \dotsc, D_r$ on $X$ if we have $Y = D_1 \cap \dotsb \cap
  D_r$ and the map $\bigoplus_{j=1}^{r} \sOX(-D_j) \longrightarrow \sI^{}_Y$
  determined by the $D_j$ is surjective.
\item On a product $Z_1 \times \dotsb \times Z_\ell$, set $E_1 \boxtimes \dotsb
  \boxtimes E_\ell := \operatorname{p}_1^{*}(E_1) \otimes \dotsb \otimes
  \operatorname{p}_\ell^{*}(E_\ell)$ where $E_j$ is a vector bundle on $j$-th
  factor $Z_j$ and $\operatorname{p}_j \colon Z_1 \times \dotsb \times Z_\ell
  \longrightarrow Z_j$ is the projection.
\end{itemize}

\section{The main vanishing theorem}
\label{sec:one}

\noindent
In this section, we prove Theorem~\ref{thm:semilocal}; Theorem~\ref{thm:main}
follows as a special case.  We begin by presenting three parts of the argument
as independent lemmas.  After completing the proof, we show that the main
theorem is optimal in some situations.

The first lemma is inspired by Proposition~2.1 in \cite{Ber}.  Following
Definition~11.1.2 in \cite{PAGII}, the asymptotic multiplier ideal
$\sJ(\norm{B})$ associated to the line bundle $B$ on $X$ is the unique maximal
member among the family of multiplier ideals $\bigl\{ \sJ\bigl( \tfrac{1}{k}
\cdot |B^{\otimes k}| \bigr) : k \in \NN \bigr\}$.

\begin{lemma}
  \label{lem:vanishing}
  Let $B$ be a big line bundle and let $E$ be an effective divisor on $X$ such
  that the line bundle $B \otimes \sOX(E)$ is big and nef.  If the closed
  subschemes in $X$ defined by the ideal sheaves $\sJ(\norm{B})$ and $\sOX(-E)$
  are disjoint, then we have $H^i( X, \KX \otimes B) = 0$ for all $i > 0$.
\end{lemma}

\begin{proof}
  For notational brevity, set $\sJ := \sJ(\norm{B})$ and $\sI := \sOX(-E)$.
  Tensoring the short exact sequence of the quotient $\sOX/\sJ$ with the ideal
  sheaf $\sI$ yields the exact sequence
  \begin{equation}
    \label{eq:fourterm}
    0 \longrightarrow \sTor^{}_1 \!\!\left( \frac{\sOX}{\sJ}, \sI
    \right) \longrightarrow \sJ \otimes \sI \longrightarrow \sI \longrightarrow
    \frac{\sI \!\!}{\sJ \!\cdot \sI} \longrightarrow 0 \, .
  \end{equation}
  Since the subschemes defined by $\sJ$ and $\sI$ are disjoint, the localization
  of $\sTor$ 
  establishes that $0 = \sTor^{}_2 \!\bigl( \sOX / \! \sJ, \sOX / \sI \, \bigr)
  = \sTor^{}_1 \!\bigl( \sOX /\! \sJ, \sI \, \bigr)$.  Hence, we have a short
  exact sequence in \eqref{eq:fourterm}, and $\sJ \otimes \sI = \sJ \!\cdot \sI
  \,$ which supplies a second short exact sequence
  \begin{equation}
    \label{eq:ses}
    0 \longrightarrow \sJ \otimes \sI \longrightarrow
    \sOX \longrightarrow \frac{\sOX \!\!}{\sJ \!\cdot \sI} \longrightarrow 0
    \, .
  \end{equation}
  Tensoring this second sequence with $A := \KX \otimes B \otimes \sOX(E)$ and
  taking cohomology produces 
  \[
  \dotsb \longrightarrow H^i\bigl( X, A \bigr) \longrightarrow H^i \!\left( X, A
    \otimes \frac{\sOX \!\!}{\sJ \!\cdot \sI} \right) \longrightarrow H^{i+1}
  \bigl( X, \KX \otimes B \otimes \sJ \bigr) \longrightarrow \dotsb \, .
  \]
  The line bundle $B \otimes \sOX(E)$ is big and nef, so the Kawamata-Viehweg
  vanishing theorem (e.g., see Theorem~4.3.1 in \cite{PAGI}) implies that $H^i
  \bigl( X, A \bigr) = 0$ for all $i > 0$.  Similarly, the line bundle $B$ is
  big, so Theorem~11.2.12~(ii) in \cite{PAGII} asserts that $H^{i+1} \bigl( X,
  \KX \otimes B \otimes \sJ \bigr) = 0$ for all $i \geq 0$.  It follows that
  $H^i \bigl( X, A \otimes (\sOX / \! \sJ \!\cdot \sI \,) \!\bigr) = 0$ for all
  $i > 0$.  Since $\sJ$ and $\sI$ have disjoint cosupport, we also have 
  $(\sOX /\! \sJ) \oplus (\sOX /\! \sI \,) \cong \sOX / \! \sJ \!\cdot \sI$,
  so $H^i \bigl( X, A \otimes (\sOX / \!  \sJ \,) \! \bigr) = 0$ for all $i >
  0$.  Moreover, the canonical inclusion of \eqref{eq:fourterm} into
  \eqref{eq:ses} proves that $\sI / \!  \sJ \!\cdot \sI \cong \sOX /\! \sJ$.
  Thus, tensoring the exact sequence in \eqref{eq:fourterm} with $A$ gives the
  short exact sequence
  \[
  0\longrightarrow \KX \otimes B \otimes \sJ \longrightarrow \KX \otimes B
  \longrightarrow A \otimes \frac{\sOX}{\sJ} \longrightarrow 0 \, . 
  \]
  Having already established that the higher cohomology groups of the left and
  right terms in this sequence vanish, we conclude that the higher cohomology
  groups of $\KX \otimes B$ also vanish.
\end{proof}

Our second lemma is a minor variant of Wilson's theorem (see Theorem~2.3.9 in
\cite{PAGI}).

\begin{lemma}
  \label{lem:Wilson}
  Let $B$ be a big, nef line bundle on $X$.  Given a coherent sheaf $\sF$ on
  $X$, there exists a positive integer $k_0$ and an effective divisor $E$ such
  that $\sF \otimes B^{\otimes k} \otimes \sOX(-E)$ is globally generated for
  all $k \geq k_0$.
\end{lemma}

\begin{proof}
  Fujita's vanishing theorem (e.g., see Theorem~1.4.35 in \cite{PAGI}) shows
  that there is a ample divisor $\widetilde{D}$ on $X$ such that $H^i \bigl( X,
  \sF \otimes \sOX(\widetilde{D}) \otimes N \bigr) = 0$ for all $i > 0$ and any
  nef line bundle $N$.  By replacing $\sOX(\widetilde{D})$ with a sufficiently
  large power, we may also assume that $\sOX(\widetilde{D})$ is globally
  generated.  Since $B$ is big, there exists a positive integer $k_0$ and an
  effective divisor $E$ on $X$ such that $B^{\otimes k_0} \cong \sOX\bigl( (d+1)
  \widetilde{D} + E\bigr)$ where $d := \dim(X)$.
  It follows that $H^i\bigl( X, \sF \otimes B^{\otimes k} \otimes \sOX(-E-i
  \widetilde{D}) \bigr) = 0$ for all $i > 0$ and all $k \geq k_0$.  Therefore,
  the sheaf $\sF \otimes B^{\otimes k} \otimes \sOX(-E)$ is $\sOX$-regular with
  respect to $\sOX(\widetilde{D})$ and globally generated (e.g., see
  Theorem~1.8.5 in \cite{PAGI}).
\end{proof}

For the remainder of this section, $m$ is a positive integer and $Y \subseteq X$
is a nonsingular subvariety of codimension $e$ defined scheme-theoretically by
the nef divisors $D_1, \dotsc, D_r$.  Following \cite{BEL} and \cite{Ber},
consider the blow-up $\pi \colon X' := \Bl_Y(X) \longrightarrow X$ of $X$ along
$Y$ and let $E := \pi^{-1}(Y)$ be the exceptional divisor.  Fix $F_j :=
\pi^*(D_j) - E$ for $1 \leq j \leq r$.  Since $Y$ lies in $D_j$, we have
$\mult^{}_Y(D_j) \geq 1$.  If $D_j^{\, \prime}$ denotes the proper transform of
$D_j$ in $X'$, then each $F_j = D_j^{\, \prime} + (\mult^{}_Y(D_j)-1) E$ is an
effective divisor.  Our third lemma constructs normal crossing divisors from
certain subsets of $F_1, \dotsc, F_r$.

\begin{lemma}
  \label{lem:technical}
  If $z \in E$ and $y := \pi(z)$, then there exists a subset $\{ s^{}_1, \dotsc,
  s^{}_e \} \subseteq \{1, \dotsc, r\}$ such that the induced map
  $\bigoplus_{j=1}^{e} \sO^{}_{X,y}(-D_{s_j}) \longrightarrow \sI^{}_{Y,y}$ is
  surjective.  Moreover, by reindexing such a subset $\{ s^{}_1, \dotsc,
  s^{}_e\}$ if necessary, we also have $z \not\in F_{s_1}$ and $F_{s_2} + \dotsb
  + F_{s_e}$ has normal crossings at $z$.
\end{lemma}

\begin{proof}
  Both $X$ and $Y$ are nonsingular, so $Y$ is a local complete intersection in
  $X$. 
  Since $Y$ is defined scheme-theoretically by $D_1, \dotsc, D_r$, it follows
  that, for every point $y \in Y$, there is a subset $\{ s^{}_1, \dotsc, s^{}_e
  \} \subseteq \{1, \dotsc, r\}$ so that the induced map $\bigoplus_{j=1}^{e}
  \sO^{}_{X,y}(-D_{s_j}) \longrightarrow \sI^{}_{Y,y}$ is surjective.

  For the second part, choose an affine neighbourhood $U\subseteq X$ of $y$ such
  that $R := \sOX(U)$ is a regular ring, the local equation of $D_{s_j}$ is $f_j
  \in R$ for $1 \leq j \leq e$, and the prime ideal $\sI^{}_{Y}(U)\subseteq R$
  is generated by $f_1, \dotsc, f_e$.  In other words, the elements $f_1,
  \dotsc, f_e$ form a regular system of parameters at the generic point of $Y
  \cap U$.  It follows that $\pi^{-1}(U)$ is the union of the open subschemes
  \[
  W_{\!  j} := \Spec \left( \frac{R[t_1, \dotsc, t_{j-1}, t_{j+1}, \dotsc,
      t_e]}{( f_1 - f_j \, t_1, \dotsc, f_{j-1} - f_j \, t_{j-1}, f_{j+1} - f_j
      \, t_{j+1}, \dotsc, f_{e} - f_j \, t_{e})} \right) \quad \text{for $1 \leq
    j \leq e$,}
  \]
  and the ideal $\sO^{}_{W_j}(-E|^{}_{W_j})$ is generated by $f_j$.  Hence, we
  see that $F_{s_j} \cap W_{\! j} = \varnothing$ and, for $k \neq j$, the local
  equation of $F_{s_k} \cap W_{\! j}$ is $\tfrac{f_k}{f_j} = t_k \in
  \sOXp(W_j)$.  By reindexing the subset $\{ s^{}_1, \dotsc, s^{}_e\}$ if
  necessary, we may assume that $z$ lies in $W_1$, so $z \not\in F_{s_1}$.
  Since the quotient $\sOXp(W_j) / (f_j)$ is the polynomial ring $k(y)[t_1,
  \dotsc t_{j-1}, t_{j+1}, \dotsc , t_e]$, we see that the divisor $F_{s_2} +
  \dotsb + F_{s_e}$ has normal crossings at $z$.
\end{proof}

With these three preliminary results, we can now prove our refinement of
Theorem~\ref{thm:main}.

\begin{theorem}
  \label{thm:semilocal}
  Fix a collection $\Phi$ of $e$-element subsets of $\{ 1, \dotsc, r\}$ such
  that, for each point $y \in Y$, there exists $\{s^{}_1, \dotsc, s^{}_e \} \in
  \Phi$ for which the induced map $\bigoplus_{j=1}^{e} \sO^{}_{X,y}(-D_{s_j})
  \longrightarrow \sI^{}_{Y,y}$ is surjective.  If the line bundle $L \otimes
  \sOX\bigl( -(m+1)D_{s_1} - D_{s_2} - \dotsb - D_{s_e} \bigr)$ is a big and nef
  for all subsets $\{s^{}_1, s^{}_2, \dotsc, s^{}_e \} \in \Phi$, then we have
  $H^i \bigl( X, \sI_{Y}^{\,\, m+1} \otimes \KX \otimes L \bigr) = 0$ for all $i
  > 0$.
\end{theorem}

\begin{remark}
  When $m > 0$, the line bundle $L \otimes \sOX\bigl( -(m+1)D_{s_1} - D_{s_2} -
  \dotsb - D_{s_e} \bigr)$ depends on the choice of the first element $s_1$ not
  just the underlying set.  In particular, the hypothesis ``for all subsets
  $\{s^{}_1, \dotsc, s^{}_e \} \in \Phi$'' means all subsets together with all
  choices of a first element $s^{}_1$.
\end{remark}

\begin{proof}
  Since $X$ and $Y$ are nonsingular, Lemma~4.3.16 in \cite{PAGI} shows that
  \[
  H^i \bigl( X, \sI_{Y}^{\; m+1} \otimes \KX \otimes L \bigr) = H^i \bigl( X',
  \pi^*( \KX ) \otimes \pi^*(L) \otimes \sOXp \bigl(- (m+1) E \bigr) \bigr) \, ,
  \quad \text{for all $i > 0$.}
  \]
  We also have $\KXp \cong \pi^*(\KX) \otimes \sOXp\bigl((e-1)E\bigr)$ because
  $Y$ is nonsingular of codimension $e$.
  Therefore, it is enough to show that $H^i \bigl( X', \KXp \otimes \pi^*(L)
  \otimes \sOXp \bigl(-(e+m)E\bigr) \!\bigr) = 0$ for all $i > 0$.  By setting
  $B := \pi^*(L) \otimes \sOXp\bigl( -(e+m)E \bigr)$, we need to show that $H^i
  \bigl( X', \KXp \otimes B \bigr) = 0$ for all $i > 0$, so it suffices to prove
  that $B$ and $(e+m)E$ satisfy the three conditions in
  Lemma~\ref{lem:vanishing}.

  The first two conditions assert that $B$ is big and that $B \otimes
  \sOXp\bigl( (e+m)E \bigr)$ is big and nef.  By hypothesis, the line bundle $L
  \otimes \sOXp(-m D_{s_1} - \sum_{j=1}^{e} D_{s_j})$ is big and nef for any $\{
  s^{}_1, \dotsc, s^{}_e \} \in \Phi$.  Hence, by decomposing $B$ as the product
  of a big line bundle and an effective line bundle
  \[
  B = \pi^*\bigl( L \otimes \sOX( - m \, D^{}_{s_1} -
  \textstyle\sum\nolimits_{j=1}^{e} D_{s_j}) \bigr) \otimes \bigl( \sOXp( m \,
  F^{}_{s_1} + \textstyle\sum\nolimits_{j=1}^{e} F_{s_j}) \bigr) \, ,
  \] 
  we see that $B$ is a big line bundle.
  Similarly, $B \otimes \sOXp\bigl((e+m)E \bigr)$ is the product of a big, nef
  line bundle, and a nef line bundle
  \[
  B \otimes \sOXp\bigl((e+m)E \bigr) = \pi^*(L) = \pi^*\bigl( L \otimes \sOX (
  -m \, D_{s_1} - \textstyle\sum\nolimits_{j=1}^{e} D_{s_j}) \bigr) \otimes
  \pi^*\bigl( \sOX( m \, D_{s_1} + \textstyle\sum\nolimits_{j=1}^{e} D_{s_j} )
  \bigr) \, ,
  \] 
  so the line bundle $B \otimes \sOXp\bigl( (e+m)E \bigr)$ is big and nef.

  To establish the third condition from Lemma~\ref{lem:vanishing}, we must show
  that the closed subschemes of $X'$ defined by $\sJ(\norm{B})$ and
  $\sOXp\bigl(-(e+m)E \bigr)$ are disjoint.  Fix a point $z \in E = \Supp\bigl(
  (e+m)E \bigr)$.  By Lemma~\ref{lem:technical}, there exists a subset
  $\{s^{}_1, \dotsc, s^{}_e \} \subseteq \Phi$ such that $z \not\in F_{s_1}$ and
  $F_{s_2} + \dotsb + F_{s_e}$ has normal crossings at $z$.  Since $\pi^*\bigl(L
  \otimes \sOX(-m \, D_{s_1} - \sum_{j=1}^{e} D_{s_j}) \bigr)$ is a big, nef
  line bundle, Lemma~\ref{lem:Wilson} provides a positive integer $k_0$ and an
  effective divisor $\widetilde{D}$ such that the linear series
  \[
  \left| \sOXp \!\bigl( m \, F_{s_1} + \textstyle\sum\nolimits_{j=1}^{e} F_{s_j}
    \bigr) \otimes \pi^*(L^{\otimes k}) \otimes \sOXp \! \bigl( - k m \,
    \pi^*(D_{s_1}) - k \textstyle\sum\nolimits_{j=1}^{e} \pi^*(D^{}_{s_j}) -
    \widetilde{D} \bigr) \right|
  \]
  has no base points for all $k \geq k_0$.  Choose an effective divisor $F'$ in
  this linear series that does not pass through $z$ and, for each $k \geq k_0$,
  consider the effective divisor 
  \[
  G := \tfrac{k-1}{k} \bigl( m \, F_{s_1} + \textstyle\sum\nolimits_{j=1}^{e}
  F_{s_j} \bigr) + \tfrac{1}{k} F' \, .
  \]  
  Since $G$
  has normal crossing support at $z$, the multiplier ideal sheaf $\sJ(G)$ is
  trivial at $z$ (cf.{} Example~9.2.13 in \cite{PAGII}).

  To complete the proof, we relate the multiplier ideal sheaf $\sJ(G)$ to the
  asymptotic multiplier ideal sheaf $\sJ(\norm{B})$.  Since $k \, G \in \bigl|
  B^{\otimes k} \otimes \sOXp(- \widetilde{D}) \bigr|$ and $\bigl| B^{\otimes k}
  \otimes \sOX(-\widetilde{D}) \bigr| \subseteq \bigl| B^{\otimes k} \bigr|$,
  Proposition~9.2.32 in \cite{PAGII} yields 
  \[
  \sJ\bigl( G \bigr) \subseteq \sJ \bigl( \tfrac{1}{k} \bigl| B^{\otimes k}
  \otimes \sOXp(-\widetilde{D}) \bigr| \bigr) \subseteq \sJ \bigl( \tfrac{1}{k}
  \bigl| B^{\otimes k} \bigr| \bigr) \, , \qquad \text{for all $k \geq k_0$.}
  \]  
  The definition of the asymptotic multiplier ideal
  implies that $\sJ \bigl( \frac{1}{k} \bigl|B^{\otimes k} \bigr|
  \bigr) \subseteq \sJ (\norm{B})$.  Hence, for sufficiently large integers $k$,
  we have $\sJ(G) \subseteq \sJ (\norm{B})$, so the asymptotic multiplier ideal
  sheaf $\sJ (\norm{B})$ must also be trivial at the point $z \in E$.  In other
  words, the closed subschemes of $X'$ defined by $\sJ(\norm{B})$ and
  $\sOXp\bigl(-(e+m)E \bigr)$ are disjoint.
\end{proof}

\begin{proof}[Proof of Theorem~\ref{thm:main}]
  Lemma~\ref{lem:technical} shows that, for any $y\in Y$, there exist a subset
  $\{s^{}_1, \dotsc, s^{}_e \} \subseteq \{1, \dotsc, r\}$ for which the induced
  map $\bigoplus_{j=1}^{e} \sO^{}_{X,y}(-D_{s_j}) \longrightarrow \sI^{}_{Y,y}$
  is surjective.  Thus, Theorem~\ref{thm:main} is simply
  Theorem~\ref{thm:semilocal} when $\Phi$ consists of all $e$-element subsets of
  $\{1, \dotsc, r\}$.
\end{proof}

\begin{remark}
  \label{rem:KawamataViehweg}
  When $m = 0$ and $r = 1$, Theorem~\ref{thm:main} specializes to the
  Kawamata-Viehweg vanishing theorem.  Indeed, if $Y \subseteq X$ is an
  effective Cartier divisor, then we have $\sI^{}_Y = \sOX(-Y)$.  Given any big
  and nef line bundle $B$ on $X$, set $L := B\otimes \sOX(Y)$.  In this case,
  Theorem~\ref{thm:main} implies that $H^i(X, \KX \otimes B) = H^i\bigl( X,
  \sI^{}_{Y} \otimes \KX \otimes L \bigr) = 0$ for all $i > 0$.  Consequently,
  one can view our main theorem as a higher-codimensional analogue of the
  Kawamata-Viehweg vanishing theorem.
\end{remark}

Although Theorem~\ref{thm:main} generalizes the Kawamata-Viehweg vanishing
theorem, the next result shows that it is essentially as sharp.  More precisely,
given a witness for the optimality of the Kawamata-Viehweg vanishing theorem, we
obtain a witness for the optimality of Theorem~\ref{thm:main}.

\begin{proposition}
  \label{pro:sharp}
  Let $i$ and $e$ be positive integers.  If $N$ is a nef line bundle on
  $X$ such that $H^{i+e-1}\bigl( X, K_X \otimes N \bigr) \neq 0$, then every
  complete intersection $Y = D^{}_1 \cap \dotsb \cap D^{}_e$ in $X$, where each
  divisor $D_i$ is big and nef, satisfies $H^i \bigl( X, \sI^{}_{Y} \otimes \KX
  \otimes \sOX(D^{}_1 + \dotsb + D^{}_e) \otimes N \bigr) \neq 0$.
\end{proposition}

\begin{proof}
  Fix a complete intersection $Y = D^{}_1 \cap \dotsb \cap D^{}_e$ in $X$ such
  that each divisor $D^{}_i$ is big and nef.  Let $\binom{[e]}{k}$ denote
  the set of all $k$-element subsets of $[e] := \{1, \dotsc, e\}$.  Tensoring
  the Koszul complex associated to $Y$ with the line bundle $\KX \otimes
  \sOX(D_1 + \dotsb + D_e) \otimes N$ yields the exact sequence
  \begin{multline*}
    0 \longrightarrow \KX \otimes N \longrightarrow \bigoplus_{\{ s_1 \} \in
      \binom{[e]}{1}} \KX \otimes \sOX(D_{s_1}) \otimes N \longrightarrow
    \bigoplus_{\{s_1,s_2\} \in \binom{[e]}{2}} \KX \otimes \sOX(D_{s_1} +
    D_{s_2}) \otimes N \longrightarrow \dotsb \\
    \dotsb \longrightarrow \bigoplus_{\{s_1, \dotsc, s_{e-1} \} \in
      \binom{[e]}{e-1}} \KX \otimes \sOX \Bigl( \textstyle\sum\limits_{j=1}^{e-1}
    D_{s_j} \Bigr) \otimes N \longrightarrow \sI^{}_{Y} \otimes \KX \otimes \sOX
    \Bigl( \textstyle\sum\limits_{j=1}^{e} D_j \Bigr) \otimes N \longrightarrow
    0 \, .
  \end{multline*}
  For any $1 \leq k \leq e$ and any nonempty subset $\{ s^{}_1, \dotsc, s^{}_k
  \} \subseteq \{1, \dotsc, e \}$, the Kawamata-Viehweg vanishing theorem shows
  that $H^i\bigl(X, \KX \otimes \sOX( \sum_{j=1}^{k} D_{s_j}) \otimes N \bigr) =
  0$ for all $i > 0$.  Chopping this exact sequence into short exact ones and
  taking the associated long exact sequences in cohomology (cf.{}
  Proposition~B.1.2 in \cite{PAGI}), we conclude that
  \[
  H^i \bigl( X, \sI^{}_{Y} \otimes \KX \otimes \sOX(D^{}_1 + \dotsb + D^{}_e)
  \otimes N \bigr) = H^{i+e-1}(X, \KX \otimes N) \neq 0 \, . \qedhere
  \]
\end{proof}

The next example shows that Theorem~\ref{thm:main} is optimal for a very general
line bundle $L$ on a product of projective spaces.

\begin{example}
  \label{exa:product}
  Let $X = \PP^{n_1} \times \dotsb \times \PP^{n_\ell}$ and write $\sOX(m_1,
  \dotsc, m_\ell) := \sO_{\!  \PP^{n_1}}(m_1) \boxtimes \dotsb \boxtimes \sO_{\!
    \PP^{n_\ell}}(m_\ell)$.  With this notation, we have $\KX = \sOX(-n_1-1,
  \dotsc, -n_\ell-1)$.  Moreover, $\sOX(m_1, \dotsc, m_\ell)$ is nef (resp.{}
  big) if and only if $m_j \geq 0$ (resp.{} $m_j > 0$) for all $1 \leq j \leq
  \ell$.  Consider, the line bundle $N = \sOX(0, m_2, \dotsc, m_\ell)$ where
  $m_j \geq n_j + 1$ for all $2 \leq j \leq \ell$.  In particular, $N$ is nef
  but not big.  The K\"unneth formula gives $H^{n_1}(X,\KX \otimes N) \neq 0$.
  Thus, for any $1 \leq e \leq n_1$ and any complete intersection $Y = D^{}_1
  \cap \dotsb \cap D^{}_e$ in $X$ where each divisor $D_j$ is big and nef,
  Proposition~\ref{pro:sharp} establishes that $H^{n_1 - e +1} \bigl( X,
  \sI^{}_{Y} \otimes \KX \otimes \sOX(D^{}_1 + \dotsb + D^{}_e) \otimes N \bigr)
  \neq 0$.  By permuting the factors of $X$, it follows that, for any line
  bundle $L$ sufficiently far into the relative interior of a facet of the nef
  cone, one cannot weaken the hypotheses on $L$ in Theorem~\ref{thm:main} and
  continue to have all of the higher cohomology groups vanish. \xqed{\Diamond}
\end{example}

For some special classes of nonsingular varieties, there are vanishing results
that improve upon the Kawamata-Viehweg vanishing theorem.  Given
Remark~\ref{rem:KawamataViehweg}, we cannot expect Theorem~\ref{thm:main} to be
uniformly optimal for all ambient varieties $X$.
 
\begin{question}
  Is it possible to strengthen Theorem~\ref{thm:main} by exploiting some
  properties of $X$?
\end{question}

\section{Multigraded Castelnuovo-Mumford regularity}

\setcounter{equation}{0}

\noindent
This section proves Corollary~\ref{cor:reg} which bounds the multigraded
Castelnuovo-Mumford regularity of $\sI^{}_Y$.  To assess this corollary, we
include a couple of examples.  Throughout the section, we measure multigraded
Castelnuovo-Mumford regularity with respect to the free divisors $P^{}_1,
\dotsc, P^{}_\ell$ on $X$ such that some positive linear combination is very
ample.

\begin{proof}[Proof of Corollary~\ref{cor:reg}]
  We must verify that $H^{i} \bigl( \sI^{}_Y \otimes \KX \otimes L \otimes N
  \otimes \sOX( -u^{}_1 \, P^{}_{1} - \dotsb - u^{}_{\ell} \, P^{}_{\ell})
  \bigr) = 0$ for all $i > 0$ and all $(u_1, \dotsc, u_\ell) \in \NN^{\ell}$
  satisfying $u_1 + \dotsb + u_{\ell} = i$.  For $i > \dim(X)$, this is
  immediate 
  and, for $i \leq \dim(Y)+1$, it follows from Theorem~\ref{thm:main}.  To
  establish the remaining vanishings, tensor the exact sequence $0
  \longrightarrow \sI^{}_Y \longrightarrow \sOX \longrightarrow \sO^{}_Y
  \longrightarrow 0$ with $A := \KX \otimes L \otimes N \otimes
  \sOX(\sum_{j=1}^{\ell} - u_j \, P_j)$ and take cohomology to obtain the exact
  sequence
  \begin{equation}
    \label{exa:longexact}
    \dotsb \longrightarrow H^{i-1} \bigl( X, \sO^{}_Y \otimes A \bigr)
    \longrightarrow H^{i} \bigl( X, \sI^{}_Y \otimes A \bigr) \longrightarrow
    H^{i} \bigl( X, A \bigr) \longrightarrow \dotsb \, .
  \end{equation}
  We have $H^{i-1} \bigl( X, \sO^{}_Y \otimes A \bigr) = 0$ for all $i - 1 >
  \dim(Y)$.  Since each divisor $D_j$ is a positive linear combination of the
  $P^{}_1, \dotsc, P^{}_\ell$, we have $-u^{}_1 \, P^{}_1 - \dotsb - u^{}_\ell
  \, P^{}_\ell = -D_{s_1} - \dotsb - D_{s_e} - u^{\prime}_1 \, P^{}_1 - \dotsb -
  u^{\prime}_\ell \, P_\ell$ where $(u'_1, \dotsc, u'_\ell) \in \NN^{\ell}$
  satisfies $u'_1 + \dotsb + u'_\ell \leq \dim(Y)$.  Hence, the line bundle
  \[
  L \otimes \sOX(\textstyle\sum\nolimits_{j=1}^{\ell} - u_j \, P_j) = L \otimes
  \sOX(- D_{s_1} - \dotsb - D_{s_e} - u^{\prime}_1 \, P^{}_1 - \dotsb -
  u^{\prime}_\ell \, P^{}_\ell)
  \] 
  is big and nef, so the Kawamata-Viehweg vanishing theorem implies that $H^i
  \bigl( X, A \bigr) = 0$ for all $i > 0$.  Therefore, we have $H^{i}
  \bigl( X, \sI^{}_Y \otimes A \bigr) = 0$ for all $\dim(Y) + 1 < i \leq
  \dim(X)$.
\end{proof}

\begin{remark}
  \label{rem:linear}
  Fix a big, nef line bundle $B$ on $X$ and consider the line bundle 
  \[
  L = B \otimes \sOX\bigl( (d-e+1) \, P_1 + \dotsb + (d-e+1) \,
  P_\ell + D_1 + \dotsb + D_r \bigr) \, ,
  \] 
  where $d := \dim(X)$ and $e := \codim(Y)$.  Corollary~\ref{cor:reg} implies
  that $\sI^{}_Y$ is $K_X \otimes L$-regular, so we have a linear bound on
  multigraded Castelnuovo-Mumford regularity for nonsingular varieties.
\end{remark}

\begin{remark}
  \label{rem:reghyp}
  Assuming that $\sOX$ is $\KX \otimes L \otimes N$-regular with respect to
  $P^{}_1, \dotsc, P^{}_\ell$, the cohomology group $H^i\bigl( X, A \bigr)$
  appearing in \eqref{exa:longexact} vanishes.  Thus, with this alternative
  hypothesis, we do not need to assume that each nef divisor $D_j$ is a positive
  linear combination of the $P^{}_1, \dotsc, P^{}_\ell$.
\end{remark}

When the ambient space is simply projective space, the following example shows
that we recover Corollary~4 in \cite{BEL} (also see Remark~1.8.44 in
\cite{PAGI}).

\begin{example}
  \label{exa:reg}
  Let $X = \PP^d$, $\ell = 1$, and $\sOX(P) = \sOX(1)$.  Since each $D_j$ is
  nef, there is a positive integer $m_j$ such that $D_j \in |m_j P|$ for all $1
  \leq j \leq r$.  We may assume that $m_1 \geq m_2 \geq \dotsb \geq m_r$.
  Thus, the line bundle 
  \[
  L = \sOX(P) \otimes \sOX \bigl( (d-e+1) \, P + D_1 + \dotsb + D_e \bigr) =
  \sOX(d - e + 2 + m_1 + \dotsb + m_e)
  \] 
  satisfies the hypothesis of Corollary~\ref{cor:reg} and $\sI^{}_Y$ is
  $\sOX(m_1 + \dotsb + m_e -e+1)$-regular. \xqed{\Diamond}
\end{example}

The final example shows, perhaps unexpectedly, that the best bounds arising from
Corollary~\ref{cor:reg} do not necessarily come from the ideal-theoretic
equations for $Y$.

\begin{example}
  \label{exa:curve}
  Fix $X = \PP^2 \times \PP^2$ as the ambient variety and choose free divisors
  $P^{}_1$ and $P^{}_2$ on $X$ such that $\sOX(P^{}_1) = \sOX(1,0)$ and
  $\sOX(P^{}_2) = \sOX(0,1)$.  Proposition~6.10 in \cite{MS} shows that $\sOX$
  is $\sOX$-regular with respect to $P^{}_1, P^{}_2$.  Let $Y \subset X$ be the
  image of the map from $\PP^1$ to $\PP^2 \times \PP^2$ given by $[t_0 : t_1]
  \mapsto \bigl( [ t_0^6 : t_0^3 t_1^3 : t_1^6], [ t_0^2 : t^{}_0 t^{}_1 :
  t_1^2] \bigr)$, so $Y$ is nonsingular of codimension $3$.  By composing this
  map with a Segre embedding of $\PP^2 \times \PP^2$ in $\PP^8$, we obtain a
  Veronese embedding of $\PP^1$ in $\PP^8$.  If $S := \CC[x_0,x_1,x_2,y_0,
  y_1,y_2]$, where $\deg(x_i) = \left[
    \begin{smallmatrix} 
      1 \\ 0 
    \end{smallmatrix} 
  \right] \in \ZZ^2$ and  $\deg(y_j) = \left[ 
    \begin{smallmatrix} 
      0 \\ 1 
    \end{smallmatrix} 
  \right] \in \ZZ^2$, is the Cox ring or total coordinate ring of $X$ and $J :=
  (x_0,x_1,x_2) \cap (y_0,y_1,y_2)$ is the irrelevant ideal, then the
  $J$-saturated $S$-ideal associated to $Y$ is $I^{}_Y = (f_0,\dotsc, f_5)$
  where
  \begin{xalignat*}{3}
    f_0 &= x_1^2 - x_0 x_2 \, , &
    f_2 &= x_2 y_0 y_1 - x_1 y_2^2 \, , &
    f_4 &= x_2 y_0^2 - x_1 y_1 y_2 \, , \\ 
    f_1 &= y_1^2 - y_0 y_2 \, , &
    f_3 &= x_1 y_0 y_1 - x_0 y_2^2 \, , &
    f_5 &= x_1 y_0^2 - x_0 y_1 y_2 \, .
  \end{xalignat*}
  Since $\deg(f_0) = \left[
    \begin{smallmatrix} 
      2 \\ 0 
    \end{smallmatrix} 
  \right]$, $\deg(f_1) = \left[
    \begin{smallmatrix} 
      0 \\ 2 
    \end{smallmatrix} 
  \right]$, and $\deg(f_2) = \dotsb = \deg(f_5) = \left[
    \begin{smallmatrix} 
      1 \\ 2 
    \end{smallmatrix} 
  \right]$, Corollary~\ref{cor:reg} together with Remark~\ref{rem:reghyp} imply
  that $\sI^{}_Y = \widetilde{I^{}_Y}$ is $\sOX(4,6)$-regular.  On the other
  hand, the intersection of prime ideal $I_Y$ and the $(y_0, y_1, y_2)$-primary
  ideal $(y_2^2, y_1 y_2, y_1^2 - y_0 y_2, y_0 y_1, y_0^2)$ equals
  $(f_1,\dotsc,f_5)$ which implies that $Y$ is defined scheme-theoretically by
  the last five equations.  Using this smaller collection of equations, we see
  that $\sI^{}_Y$ is $\sOX(3,6)$-regular.  However, Theorem~A in \cite{Loz}
  establishes that $\sI^{}_Y$ is also $\sOX(1,5)$-regular because the curve $Y$
  has bidegree $(6,2)$.  \xqed{\Diamond}
\end{example}

Example~\ref{exa:curve} also shows that better bounds can be obtained by using
additional numerical invariants.  The following seems like an excellent place to
start exploring this phenomena.

\begin{question}
  Suppose that $Y$ is a nonsingular curve embedded in a nonsingular toric
  variety $X$.  What are the optimal bounds on the multigraded
  Castelnuovo-Mumford regularity of $Y$ in terms of its multidegree?
\end{question}

\section{Maps arising from the global sections of adjoint bundles}

\noindent
The goal of this section is to provide sufficient conditions for the
surjectivity of both the canonical multiplication map between the global
sections of adjoint bundles and the Wahl maps arising from adjoint bundles.  Let
$A^{}_1, \dotsc, A^{}_d$ denote very ample line bundles on $X$ where $d :=
\dim(X)$.

\begin{proof}[Proof of Corollary~\ref{cor:veryample}]
  Tensoring the short exact sequence for the diagonal $\Delta(X) \subseteq X
  \times X$ with $(\KX \otimes L^{}_1) \boxtimes (\KX \otimes L^{}_2) = K^{}_{X
    \times X} \otimes (L^{}_1 \boxtimes L^{}_2)$ yields the exact sequence
  \[
  0 \longrightarrow \sI^{}_{\!  \Delta(X)}\otimes K^{}_{X \times X} \otimes
  (L^{}_1 \boxtimes L^{ }_2) \longrightarrow K^{}_{X \times X} \otimes (L^{}_1
  \boxtimes L^{}_2) \longrightarrow \sO^{}_{\!  \Delta(X)} \otimes K^{}_{X
    \times X} \otimes ( L^{}_1 \boxtimes L^{}_2) \longrightarrow 0 \, .
  \]
  By restricting to the diagonal $\Delta(X)$ and taking global sections, we
  obtain the multiplication map $H^{0}(X, \KX \otimes L^{}_1) \otimes H^{0}(X,
  \KX \otimes L^{}_2) \longrightarrow H^{0}(X, K_X^{\otimes 2} \otimes L^{}_1
  \otimes L^{}_2)$.  Therefore, it suffices to show that $H^{1} \bigl( X \times
  X, \sI^{}_{\!  \Delta(X)} \otimes K^{}_{X\times X} \otimes (L^{}_1 \boxtimes
  L^{}_2) \bigr) = 0$.

  Since each $A_k$ is very ample, the diagonal $\Delta(X) \subseteq X \times X$
  is defined scheme-theoretically by the complete linear series $|A_k \boxtimes
  A_k|$ (cf.{} Lemma~\ref{lem:diagonal}).  As $\Delta(X)$ and $X \times X$ are
  both nonsingular, we can choose, for each point $x \in \Delta(X)$ and each $1
  \leq k \leq d$, a single divisor $P_k \in | A_k \boxtimes A_k |$ such that the
  induced map $\bigoplus_{k=1}^d \sO^{}_{X \times X, x}(-P_k) \rightarrow
  \sI^{}_{\Delta(X),x}$ is surjective.  Thus, the vanishing follows from
  Theorem~\ref{thm:semilocal} because
  \begin{multline*}
    (L^{}_1 \boxtimes L^{}_2) \otimes \bigl( \sOX(-P^{}_{1})
    \boxtimes \sOX(-P^{}_{1}) \bigr) \otimes \dotsb \otimes \bigl(
    \sOX(-P^{}_{d}) \boxtimes \sOX(-P^{}_{d}) \bigr) \\ = \bigl( L^{}_1
    \otimes A^{-1}_{1} \otimes \dotsb \otimes A^{-1}_{d} \bigr) \boxtimes \bigl(
    L^{}_2 \otimes A^{-1}_{1} \otimes \dotsb \otimes A^{-1}_{d} \bigr)
  \end{multline*}
  is big and nef.
\end{proof}

\begin{remark}
  Corollary~\ref{cor:veryample} specializes to Variant~3.2 in \cite{BEL} when
  $A_1 = \dotsc = A_d$.  On the other hand, Lazarsfeld points out that one can
  obtain another proof for Corollary~\ref{cor:veryample} by adapting the
  argument for Variant~3.2 in \cite{BEL}.
\end{remark}

To determine if an adjoint bundle is very ample, we have the following simple
observation.

\begin{corollary}
  \label{cor:normal}
  If $A^{}_{1}, \dotsc, A^{}_{d+2}$ are very ample line bundles on $X$ and $N$
  is a nef line bundle on $X$, then the line bundle $\KX \otimes A^{}_{1}
  \otimes \dotsb \otimes A^{}_{d+2} \otimes N$ is very ample and defines a
  projectively normal embedding.
\end{corollary}

\begin{proof}
  By applying Corollary~\ref{cor:veryample}, it is enough to show that $\KX
  \otimes A_{1} \otimes \dotsb \otimes A_{d+2} \otimes N$ is very ample.  Since
  the tensor product of a globally generated line bundle and a very ample one is
  also very ample, it suffices to prove that $B := \KX \otimes A_{1} \otimes
  \dotsb \otimes A_{d+1} \otimes N$ is globally generated.  To accomplish this,
  we proceed by induction on $d := \dim(X)$.

  Suppose that $d = 1$.  On a smooth curve $C$ of genus $g$, a divisor is ample
  (nef) if and only if it has positive (nonnegative) degree and a divisor of
  degree at least $2g$ has no base point.
  Since $\deg(K^{}_C) = 2g-2$, it follows that $K^{}_C \otimes A_1 \otimes A_2
  \otimes N$ is globally generated.

  Now suppose that $d > 1$ and fix a point $x \in X$.  The line bundle $A_{d+1}$
  is very ample, so Bertini's Theorem 
  provides a smooth irreducible divisor $H \in |A_{d+1}|$ passing through $x$.
  Tensoring the short exact sequence $0 \longrightarrow \sOX(-H) \longrightarrow
  \sOX \longrightarrow \sO^{}_H \longrightarrow 0$ with $B$ yields
  \[
  0 \longrightarrow B \otimes A^{-1}_{d+1} \longrightarrow B \longrightarrow
  K^{}_H \otimes A^{}_1 \otimes \dotsb \otimes A^{}_{d} \otimes N
  \longrightarrow 0 \, .
  \]
  Since $B \otimes A^{-1}_{d+1} = \KX \otimes A_{1} \otimes \dotsb \otimes A_{d}
  \otimes N$, the Kawamata-Viehweg vanishing theorem shows that $H^1\bigl( X, B
  \otimes A^{-1}_{d+1} \bigr) = 0$.  Hence, the map $H^{0} \bigl( X, B \bigr)
  \longrightarrow H^{0} \bigl(H, K^{}_H \otimes A^{}_1 \otimes \dotsb \otimes
  A^{}_{d} \otimes N \bigr)$ is surjective.  The induction hypothesis ensures
  that there exists a global section of $K^{}_H \otimes A^{}_1 \otimes \dotsb
  \otimes A^{}_{d} \otimes N$ that does not vanish at the point $x$.  By
  choosing a preimage in $H^{0} \bigl( X, B \bigr)$, we obtain a global section
  of $B$ that does not vanish at $x$.
\end{proof}

To give alternative conditions for the surjectivity of the multiplication map,
fix globally generated line bundles $\sOX(P^{}_1), \dotsc, \sOX(P^{}_\ell)$ on
$X$ such that the associated complete linear series define a closed embedding $X
\longrightarrow \PP^{n_1} \times \dotsb \times \PP^{n_\ell}$.  Loosely speaking,
we view the line bundles $\sOX(P^{}_1), \dotsc, \sOX(P^{}_\ell)$ as a reasonable
factorization of the very ample line bundle $\sOX(P^{}_1 + \dotsb + P^{}_\ell)$.

\begin{proposition}
  \label{pro:adjoint}
  If $L^{}_1$ and $L^{}_2$ are line bundles on $X$ such that $L^{}_j \otimes
  \sOX(-u^{}_1 \, P^{}_{1} - \dotsb - u^{}_\ell \, P^{}_{\ell})$ is big and nef
  for all $j$ and all $(u_1, \dotsc, u_\ell) \in \NN^\ell$ such that $u_1 +
  \dotsb + u_\ell = d$ and $0 \leq u_k \leq \binom{n_k +1}{2}$, then the
  multiplication map $H^{0} \bigl( X, \KX \otimes L^{}_1 \bigr) \otimes H^{0}
  \bigl( X, \KX \otimes L^{}_2 \bigr) \longrightarrow H^{0} \bigl(X,
  K_X^{\otimes 2} \otimes L^{}_1 \otimes L^{}_2 \bigr)$ is surjective.
\end{proposition}

Before proving this proposition, we record a multigraded extension of Lemma~3.1
in \cite{BEL}.

\begin{lemma}
  \label{lem:diagonal}
  The diagonal $\Delta(X) \subseteq X \times X$ is defined scheme-theoretically
  by divisors obtained from $\binom{n_k+1}{2}$ sections of $\sOX(P_k) \boxtimes
  \sOX(P_k)$ for each $1\leq k \leq \ell$.
\end{lemma}

\begin{proof}
  Since the line bundles $\sOX(P_1), \dotsc, \sOX(P_\ell)$ give a closed
  embedding $X \subseteq \PP^{n_1} \times \dotsb \times \PP^{n_\ell}$, we have
  $X \times X \subseteq \bigl( \PP^{n_1} \times \dotsb \times \PP^{n_\ell}
  \bigr) \times \bigl( \PP^{n_1} \times \dotsb \times \PP^{n_\ell} \bigr)$ and
  $\Delta(X) = \Delta(\PP^{n_1} \times \dotsb \times \PP^{n_\ell}) \cap (X
  \times X)$.  Each $\Delta(\PP^{n_k}) \subseteq \PP^{n_k} \times \PP^{n_k}$ is
  defined ideal-theoretically by $\binom{n_k+1}{2}$ sections of
  $\sO^{}_{\PP^{n_k}}(1) \boxtimes \sO^{}_{\PP^{n_k}}(1)$ (e.g., see
  Exercise~13.15.b in \cite{Eis}).  Thus, the diagonal $\Delta(\PP^{n_1} \times
  \dotsb \times \PP^{n_\ell})$ is defined scheme-theoretically by entire
  collection of sections.  The claim then follows by pulling-back to $X \times
  X$.
\end{proof}

\begin{proof}[Proof of Proposition~\ref{pro:adjoint}]
  Just as in the proof of Corollary~\ref{cor:veryample}, it suffices to show
  that 
  \[
  H^{1} \bigl( X \times X, \sI^{}_{\!  \Delta(X)} \otimes K^{}_{X\times X}
  \otimes (L^{}_1 \boxtimes L^{}_2) \bigr) = 0 \, .
  \]
  Lemma~\ref{lem:diagonal} shows that $\Delta(X)$ is defined
  scheme-theoretically by divisors obtained from $\binom{n_k+1}{2}$ sections of
  $\sOX(P^{}_k) \boxtimes \sOX(P^{}_k)$ for $1 \leq k \leq \ell$.  Moreover, the
  divisor
  \begin{multline*}
    (L^{}_1 \boxtimes L^{}_2) \otimes \bigl( \sOX(-P^{}_{1}) \boxtimes
    \sOX(-P^{}_{2}) \bigr)^{\otimes u_1} \otimes \dotsb \otimes \bigl(
    \sOX(-P^{}_{\ell}) \boxtimes \sOX(-P^{}_{\ell}) \bigr)^{\otimes u_\ell} \\=
    \bigl( L^{}_1 \otimes \sOX(-u^{}_1 \, P^{}_{1} - \dotsb - u^{}_\ell \,
    P^{}_{\ell}) \bigr) \boxtimes \bigl( L^{}_2 \otimes \sOX(-u^{}_1 \, P^{}_{1}
    - \dotsb - u^{}_{\ell} \, P^{}_{\ell}) \bigr)
  \end{multline*}
  is big and nef for each $(u_1, \dotsc, u_\ell) \in \NN^\ell$ such that $u_1 +
  \dotsb + u_\ell = d$ and $0 \leq u_k \leq \binom{n_k +1}{2}$.  Therefore, the
  required vanishing follows from Theorem~\ref{thm:main}.
\end{proof}

\begin{remark}
  \label{rem:surjective}
  Whenever $\ell = 1$, we again recover Variant~3.2 in \cite{BEL} from
  Proposition~\ref{pro:adjoint}. 
\end{remark}

The first example in this section demonstrates that, for some varieties,
Proposition~\ref{pro:adjoint} can yield stronger results than
Corollary~\ref{cor:veryample} or Variant~3.2 in \cite{BEL}.

\begin{example}
  \label{exa:comparative}
  Let $X$ be the blow-up of $\PP^d$ along the intersection of two hyperplanes.
  In other words, $X$ is the toric variety obtain by blowing-up $\PP^d$ along a
  codimension-two torus orbit closure (e.g., see Proposition~3.3.15 in
  \cite{CLS}).  If $E$ is the exceptional divisor and $H$ is the proper
  transform of a hyperplane, then $\sOX(E)$ and $\sOX(H)$ form a basis for
  $\Pic(X)$ with $\KX = \sOX\bigl( -(d+1)H + E \bigr)$.  The numerical classes
  of $P^{}_1 := H-E$ and $P^{}_2 := H$ generate the nef cone $\Nef(X)$, and the
  numerical classes $[E]$ and $[P^{}_1]$ generated the pseudoeffective cone of
  $X$.  Moreover, the complete linear series associated to $\sOX(P^{}_1)$ and
  $\sOX(P^{}_2)$ define a closed embedding $X \longrightarrow \PP^1 \times
  \PP^d$; in particular, we have $n_1+1 = \dim H^0\bigl(X, \sOX(P^{}_1) \bigr) =
  2$ and $n_2 +1 = \dim(H^0\bigl(X,\sOX(P^{}_2) \bigr) = d+1$.  Hence,
  Proposition~\ref{pro:adjoint} shows that the multiplication map $H^0\bigl( X,
  L_1 \bigl) \otimes H^0 \bigl( X , L_2 \bigr) \longrightarrow H^0 \bigl( X, L_1
  \otimes L_2 \bigr)$ is surjective provided that each $L_j$ has the form
  $\sOX(P^{}_2) \otimes N$ for some nef line bundle $N$ (see the shaded-region
  in Figure~\ref{fig:blowup}).  In contrast, Corollary~\ref{cor:veryample} or
  Variant~3.2 in \cite{BEL} apply only to adjoint bundles of the form
  $\sOX\bigl( (d-1)P^{}_1\bigr) \otimes N$ where $N$ is a nef line
  bundle. \xqed{\Diamond}
\end{example}

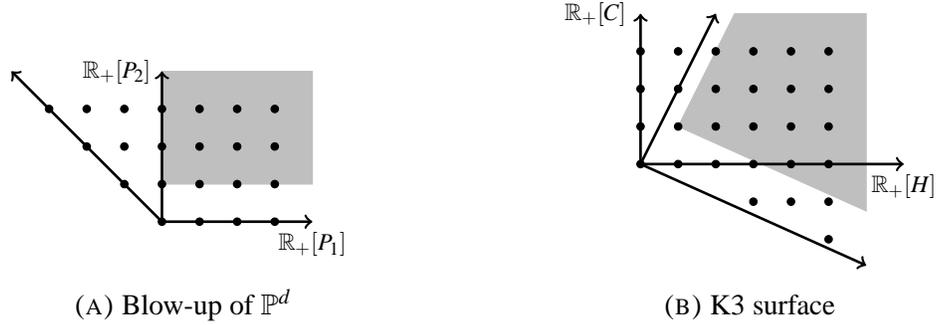
\begin{figure}[t]
  \begin{subfigure}[b]{0.45\textwidth}
    \begin{center}
      \begin{tikzpicture}[scale=1]
        \filldraw[color=lightgray] (0,0.5) -- (2,0.5) -- (2,2) -- (0,2) --
        (0,0.5);
        \draw [->, line width=1pt] (0,0) -- (-2,2);
        \draw [->, line width=1pt] (0,0) -- (0,2);
        \draw [->, line width=1pt] (0,0) -- (2,0);   
        \foreach \y in {0,...,3} {
          \foreach \x in { -\y,...,3} {
            \node [circle, draw, color=black, fill=black, inner sep=1pt] at
            (\x*0.5,\y*0.5) {};}}
        \node[color=black] at (2,-0.3) {\scriptsize $\RR_+ [P_1]$};
        \node[color=black] at (-0.6,2) {\scriptsize $\RR_+ [P_2]$};
      \end{tikzpicture}
    \end{center}
    \caption{Blow-up of $\PP^d$}
    \label{fig:blowup}
  \end{subfigure}
  \begin{subfigure}[b]{0.45\textwidth}
    \begin{center}
      \begin{tikzpicture}[scale=1]
        \filldraw[color=lightgray] (0.5,0.5) -- (1.25,2) -- (3,2) -- 
        (3,-0.449489742783178*2.5+0.5) -- (0.5,0.5);
        \draw [->, line width=1pt] (0,0) -- (3,-0.449489742783178*3);
        \draw [->, line width=1pt] (0,0) -- (0,2);
        \draw [->, line width=1pt] (0,0) -- (1,2);
        \draw [->, line width=1pt] (0,0) -- (3.5,0);   
        \foreach \y in {0,...,3} {
          \foreach \x in {0,...,5} {
            \node [circle, draw, color=black, fill=black, inner sep=1pt] at
            (\x*0.5,\y*0.5) {};}}
        \node [circle, draw, color=black, fill=black, inner sep=1pt] at
        (1.5,-0.5) {};
        \node [circle, draw, color=black, fill=black, inner sep=1pt] at
        (2,-0.5) {};
        \node [circle, draw, color=black, fill=black, inner sep=1pt] at
        (2.5,-0.5) {};
        \node [circle, draw, color=black, fill=black, inner sep=1pt] at
        (2.5,-1) {};
        \node[color=black] at (3.5,-0.3) {\scriptsize $\RR_+ [H]$};
        \node[color=black] at (-0.6,2) {\scriptsize $\RR_+ [C]$};
      \end{tikzpicture}
    \end{center}
    \caption{K3 surface}
    \label{fig:K3}
  \end{subfigure}
  \caption{N\'eron-Severi groups}
\end{figure}

Our second example shows the reverse: Corollary~\ref{cor:veryample} can be
stronger than Proposition~\ref{pro:adjoint}.

\begin{example}
  \label{exa:K3}
  Let $X \subseteq \PP^3$ be a nonsingular quartic surface that contains a
  rational quartic curve $C$ (for existence, see Theorem~1 in \cite{Mori}).  Let
  $H$ be a plane section of $X$.  Since $X$ is a K3 surface, we have $\KX =
  \sOX$.  Remark~4 in \cite{Mori} indicates that $H^2 = 4$, $H.C = 4$, $C^2 =
  -2$, and $\Pic(X) = \ZZ H + \ZZ C$.  In addition, Theorem~2 in \cite{Kov}
  implies that the cone of curves on $X$ is generated by numerical classes of
  $C$ and $H + (2 - \sqrt{6})C$.  Since $(H+2C).C = 0$, we deduce that the nef
  cone $\Nef(X)$ is generated by numerical classes $[H+2C]$ and $[H +
  (2-\sqrt{6})C]$.  Using results of Saint-Donat (e.g., see Theorem~5 in
  \cite{Mori}), we also see that $H$, $H + C$ are very ample and $H+2C$ has no
  base points.  For $H$ and $H+C$, Corollary~\ref{cor:veryample} shows that the
  multiplication map $H^0\bigl( X, L_1 \bigl) \otimes H^0 \bigl( X , L_2 \bigr)
  \longrightarrow H^0 \bigl( X, L_1 \otimes L_2 \bigr)$ is surjective provided
  that both $[L_j]$ lies in the cone $Q = [2H+C] + \Nef(X)$ (see the
  shaded-region in Figure~\ref{fig:K3}).  In contrast, applying
  Proposition~\ref{pro:adjoint} to $P^{}_1 := H+2C$ and $P^{}_2 := H+C$ only
  shows that the multiplication map is surjective when each $[L_j]$ lies in the
  cone
  \begin{align*}
    & \relphantom{=} \bigl( [2P^{}_1] + \Nef(X) \bigr) \cap \bigl( [P^{}_1 +
    P^{}_2] + \Nef(X) \bigr) \cap \bigl( [2P^{}_2] + \Nef(X) \bigr) \\ &= \Bigl[
    \Bigl(2 + \tfrac{2}{\sqrt{6}} \Bigr) H + \Bigl(2 + \tfrac{4}{\sqrt{6}}
    \Bigr) C \Bigr] + \Nef(X) \subseteq Q \, .
  \end{align*}
  In fact, no matter how we choose of the divisors $P_1, \dotsc, P_\ell$,
  Proposition~\ref{pro:adjoint} will not improve upon
  Corollary~\ref{cor:veryample} in this case.  \xqed{\Diamond}
\end{example}

By combining Example~\ref{exa:comparative} and Example~\ref{exa:K3}, we see that
the relative strengths of Corollary~\ref{cor:veryample} and
Proposition~\ref{pro:adjoint} are incomparable.  However, there is an obvious
common refinement in which one factors each very ample line bundle into
appropriate globally generated line bundles; we leave the details to the
interested reader.
  
Although one can use the ideas from Corollary~\ref{cor:veryample} or
Proposition~\ref{pro:adjoint} together with the methods in \cite{Ina} to say
something about the higher syzygies of $X$, we expect these results to be far
from optimal.  We suspect that following question points in a more fruitful
direction.

\begin{question}
  Is there an analogue of Theorem~1 in \cite{EL} which mirrors either
  Corollary~\ref{cor:veryample} or Proposition~\ref{pro:adjoint}?
\end{question}

To exhibit the utility of the parameter $m$ in our main theorem, we conclude
this section with the following result.

\begin{proposition}
  \label{pro:Wahl}
  If $L^{}_1$ and $L^{}_2$ are line bundles on $X$ such that each $L^{}_j
  \otimes \sOX(-u^{}_1 \, P^{}_{1} - \dotsb - u_\ell \, P^{}_{\ell})$ is big and
  nef for all $j=1,2$ and all $(u_1, \dotsc, u_\ell) \in \NN^\ell$ with $u_1 +
  \dotsb + u_\ell = d+m$ and $0 \leq u_k \leq \binom{n_k +1}{2}$, then the
  $m$-th Wahl map
  \[
  H^{0} \bigl( X \times X, \sI^{\; m}_{\Delta(X)} \otimes K^{}_{X\times X}
  \otimes (L^{}_1 \boxtimes L^{}_2) \bigr) \longrightarrow H^{0} \bigl( X,
  \Sym^m (\Omega_{X}^1) \otimes K^{\otimes 2}_X \otimes L^{}_1 \otimes L^{}_2
  \bigr)
  \] 
  is surjective.
\end{proposition}

\begin{proof}
  Following Section~1 in \cite{Wahl}, the $m$-th Wahl maps is defined as
  follows.  For brevity, set $\sI := \sI^{}_{\! \Delta(X)}$.  Since $\sI/\sI^{\;
    2}$ is isomorphic to $\Omega_X^1$ as an $\sO^{}_{\Delta(X)}$-module and
  $\sI^{\; m} / \sI^{\; m+1} \cong \Sym^m(\Omega_X^1)$, there is an exact
  sequence $0 \longrightarrow \sI^{\; m+1} \longrightarrow \sI^{\; m}
  \longrightarrow \Sym^m (\Omega_X^1) \otimes \sO^{}_{\!  \Delta(X)}
  \longrightarrow 0$.  Tensoring with $K^{}_{X \times X} \otimes (L^{}_1
  \boxtimes L^{}_2) = (\KX \otimes L^{}_1) \boxtimes (\KX \otimes L^{}_2)$ and
  taking global sections yields the $m$-th Wahl map
  \[
  H^{0} \bigl( X \times X, \sI^{\; m} \otimes K^{}_{X \times X} \otimes (L^{}_1
  \boxtimes L^{}_2) \bigr) \longrightarrow H^{0} \bigl( X, \Sym^m (\Omega_{X}^1)
  \otimes K^{\otimes 2}_X \otimes L^{}_1 \otimes L^{}_2 \bigr) \bigr) \, .
  \]
  To show that it is surjective, it suffices to prove that $H^{1} \bigl( X
  \times X, \sI^{\; m+1} \otimes K^{}_{X \times X} \otimes (L^{}_1 \boxtimes
  L^{}_2) \bigr)$ vanishes which follows by combining Lemma~\ref{lem:diagonal}
  and Theorem~\ref{thm:main}.
\end{proof}

\begin{remark}
  We recover Corollary~3.4 in \cite{BEL} from Proposition~\ref{pro:Wahl} when
  $m=1$, $\ell = d+m$ and $P_1 = \dotsb = P_{d+m}$.  The version of
  Proposition~\ref{pro:Wahl} that mimics Corollary~\ref{cor:veryample} is left
  to the interested reader.
\end{remark}

\section{Vanishing for Vector Bundles}

\noindent
This section presents the proof of our Griffiths-type vanishing theorem for
vector bundles.  Throughout, $E$ denotes a vector bundle of rank $e$ on $X$
which is a quotient of a direct sum of line bundles $\sOX(D^{}_1) \oplus \dotsb
\oplus \sOX(D^{}_r) \longrightarrow E$ where each divisor $D_i$ is nef.

\begin{proof}[Proof of Proposition~\ref{pro:Griffiths}]
  Let $\pi \colon X' := \PP(E) \longrightarrow X$ be the projective bundle of
  one-dimensional quotients of $E$ and let $\sOXp(1)$ be the tautological line
  bundle.  From the cotangent bundle sequence for $\pi$ and the relative Euler
  sequence, we obtain $\KXp = \pi^*\bigl( \KX \otimes \det(E) \bigr) \otimes
  \sOXp(-e)$,
  and the projection formula 
  gives $\pi^{}_* \bigl( \KXp \otimes \sOXp(m+e) \otimes \pi^*(L) \bigr) = \KX
  \otimes \det(E) \otimes \Sym^m (E) \otimes L$.
  Since $m \geq 0$, we have $\mathbf{R}^i \pi^{}_* \bigl( \sOXp(m) \bigr) = 0$
  for all $i > 0$
  , so 
  \[
  H^{i} \bigl( X', \KXp \otimes \sOXp(m+e) \otimes \pi^*(L) \bigr) = H^{i}
  \bigl( X, \KX \otimes \det(E) \otimes \Sym^m (E) \otimes L \bigr) \, , \quad
  \text{for all $i$.}
  \] 
  Setting $B := \pi^*(L) \otimes \sOXp(m+e)$, it suffices to show that $H^{i}
  \bigl( X', \KXp \otimes B) = 0$ for all $i > 0$.

  Since $E$ is as a quotient of $\bigoplus_{j=1}^r \sOX(D_j)$, there is a map
  $\sOX(D_j) \longrightarrow E$ for each $1 \leq j \leq r$.  Let $F_j \in \bigl|
  \sOXp(1) \otimes \sOXp\bigl(- \pi^*(D_j) \bigr) \bigr|$ be the effective
  divisor on $X'$ associated to the global section $\sOX \longrightarrow
  \Sym^1(E) \otimes \sOX(-D_i)$.  We see that $B$ is big from the decomposition
  \[
  B = \pi^* \bigl( L \otimes \sOXp ( m \, D_1 + \textstyle\sum\nolimits_{j=1}^e
  D_j) \bigr) \otimes \sOXp( m \, F_1 + \textstyle\sum\nolimits_{j=1}^{e} F_j )
  \, .
  \] 
  Hence, Theorem~11.2.12 (ii) in \cite{PAGII} implies that $H^i \bigr(X', \KXp
  \otimes B \otimes \sJ(\norm{B}) \bigr) = 0$ for all $i > 0$ and we need only
  show that $\sJ(\norm{B}) = \sOXp$.

  To achieve this, we claim that, for each $z \in X'$, there is a subset
  $\{s^{}_1, \dotsc, s^{}_e\} \subseteq \{1, \dotsc, r\}$ such that $z \not\in
  F_{s_1}$ and $F_{s_2} + \dotsb + F_{s_e}$ has normal crossings at $z$ (cf.{}
  Lemma~\ref{lem:technical}).  Indeed, there exists an affine neighbourhood $U$
  of $\pi(z)$ such that $\pi^{-1}(U) \cong \Spec \bigl( \sOX(U) [t_1, \dotsc,
  t_e] \bigr)$.  By shrinking $U$ if necessary, we see that there is a subset
  $\{ s^{}_1, \dotsc, s^{}_e \} \subseteq \{1, \dotsc, r\}$ such that the local
  equations corresponding to the sections $\sOX \longrightarrow \Sym^1(E)
  \otimes \sOX(- D_{s_j})$ form a regular system of parameters in the ring
  $\sOX(U) [t_1, \dotsc, t_e]$.  Hence, we may assume that $z \not\in F_{s_1}$.
  Since $\sOX(U)[t_1, \dotsc, t_e]$ is a polynomial ring, we also deduce that
  $F_{s_2} + \dotsb + F_{s_e}$ has normal crossings at $z$.

  Now, the line bundle $\pi^*\bigl(L \otimes \sOX( m \, D_{s_1} + \sum_{j=1}^{e}
  D_{s_j}) \bigr)$ is a big and nef, so Lemma~\ref{lem:Wilson} yields a positive
  integer $k_0$ and an effective divisor $\widetilde{D}$ on $X'$ such that the
  linear series
  \[
  \left| \sOXp(m \, F_{s_1} + \textstyle\sum\nolimits_{j=1}^e F_{s_j}) \otimes
    \pi^*(L^k) \otimes \sOXp \bigl( km \, \pi^*(D_{s_1}) + k \,
    \textstyle\sum\nolimits_{j=1}^{e} \pi^*(D_{s_j}) - \widetilde{D} \bigr)
  \right|
  \]
  has no base points for all $k \geq k_0$.  Choose an effective divisor $F'$ in
  this linear series that does not pass through $z$, and for each $k \geq k_0$,
  consider the effective divisor $G := \frac{k-1}{k} \bigr( m \, F_{s_1} +
  \sum\nolimits_{j=1}^e F_{s_j} \bigr) + \frac{1}{k} F'$.  Since $G$ has normal
  crossing support at $z$, $\sJ(G)$ is trivial at $z$.
  Finally, we relate the multiplier ideal sheaf $\sJ(G)$ to the asymptotic
  multiplier ideal sheaf $\sJ(\norm{B})$ as in the proof of
  Theorem~\ref{thm:main}. \qedhere
\end{proof}

We end with an uncomplicated example illustrating
Proposition~\ref{pro:Griffiths}.

\begin{example}
  Let $X = \PP^{n_1} \times \dotsb \times \PP^{n_\ell}$.  If the divisor $D_k$
  is the pull-back of the hyperplane from the $k$-factor $\PP^{n_k}$, then the
  Euler sequence (e.g., see Theorem~8.1.6 in \cite{CLS}) implies that the
  tangent bundle $T^{}_X$ is a quotient of the direct sum
  $\bigoplus_{k=1}^{\ell} \bigoplus_{j=1}^{n_k+1} \sOX(D_k)$.  Since each line
  bundle $\sOX(D_k)$ is nef and $\det(T^{}_X)$ is the anticanonical line bundle
  on $X$, Proposition~\ref{pro:Griffiths} shows that $H^i\bigl( X,
  \Sym^m(T^{}_X) \otimes B \bigr) = 0$ for all $i > 0$, all $m \geq 0$, and all
  big, nef line bundles $B$.  If $n_j \geq \ell$ for all $1 \leq j \leq \ell$,
  then $H^i\bigl( X, \Sym^m(T^{}_X) \otimes N \bigr) = 0$ for all $i > 0$, all
  $m \geq 0$, and all nef line bundles $N$.

\end{example}

\section*{Acknowledgements}

\noindent
We thank Tommaso de Fernex, Rob Lazarsfeld, and Mike Roth for helpful
discussions.  Both authors were partially supported by NSERC.


\raggedright


\begin{bibdiv}
\begin{biblist}

\bib{Ber}{article}{
  label={Ber},
  author={Bertram, Aaron},
  title={An application of a log version of the Kodaira
    vanishing theorem to embedded projective varieties},
  status={available at \href{http://arxiv.org/abs/alg-geom/9707001}%
    {\texttt{arXiv:alg-geom/9707001v1}}}
}

\bib{BEL}{article}{
  author={Bertram, Aaron},
  author={Ein, Lawrence},
  author={Lazarsfeld, Robert},
  title={\href{http://dx.doi.org/10.2307/2939270}%
    {Vanishing theorems, a theorem of Severi, and the}  
    \href{http://dx.doi.org/10.2307/2939270}%
    {equations defining projective varieties}},
  journal={J. Amer. Math. Soc.},
  volume={4},
  date={1991},
  number={3},
  pages={587--602},
}

\bib{CU}{article}{
  author={Chardin, Marc},
  author={Ulrich, Bernd},
  title={\href{http://www.jstor.org/stable/25099154}%
    {Liaison and Castelnuovo-Mumford regularity}},
  journal={Amer. J. Math.},
  volume={124},
  date={2002},
  number={6},
  pages={1103--1124},
}
	
\bib{CK}{book}{
  author={Cox, David A.},
  author={Katz, Sheldon},
  title={Mirror symmetry and algebraic geometry},
  series={Mathematical Surveys and Monographs},
  volume={68},
  publisher={American Mathematical Society},
  place={Providence, RI},
  date={1999},
  pages={xxii+469},
}

\bib{CLS}{book}{
  author={Cox, David A.},
  author={Little, John B.},
  author={Schenck, Henry K.},
  title={Toric varieties},
  series={Graduate Studies in Mathematics},
  volume={124},
  publisher={American Mathematical Society},
  place={Providence, RI},
  date={2011},
}

\bib{EL}{article}{
  author={Ein, Lawrence},
  author={Lazarsfeld, Robert},
  title={\href{http://dx.doi.org/10.1007/BF01231279}%
    {Syzygies and Koszul cohomology of smooth projective varieties of}
    \href{http://dx.doi.org/10.1007/BF01231279}%
    {arbitrary dimension}},
  journal={Invent. Math.},
  volume={111},
  date={1993},
  number={1},
  pages={51--67},
}

\bib{Eis}{book}{
  label={Eis},
  author={Eisenbud, David},
  title={Commutative algebra with a view toward algebraic geometry},
  series={Graduate Texts in Mathematics},
  volume={150},
  publisher={Springer-Verlag},
  place={New York},
  date={1995},
  pages={xvi+785},
}

\bib{FE}{article}{
  label={dFE},
  author={de Fernex, Tommaso},
  author={Ein, Lawrence},
  title={\href{http://dx.doi.org/10.1353/ajm.2010.0008}%
    {A vanishing theorem for log canonical pairs}},
  journal={Amer. J. Math.},
  volume={132},
  date={2010},
  number={5},
  pages={1205--1221},
}


\bib{HSS}{article}{
  author={Hering, Milena},
  author={Schenck, Hal},
  author={Smith, Gregory G.},
  title={\href{http://dx.doi.org/10.1112/S0010437X0600251X}%
    {Syzygies, multigraded regularity and toric varieties}},
  journal={Compos. Math.},
  volume={142},
  date={2006},
  number={6},
  pages={1499--1506},
}

\bib{HK}{article}{
  author={Hu, Yi},
  author={Keel, Sean},
  title={\href{http://dx.doi.org/10.1307/mmj/1030132722}%
    {Mori dream spaces and GIT}},
  journal={Michigan Math. J.},
  volume={48},
  date={2000},
  pages={331--348},
}

\bib{Kov}{article}{
  label={Kov},
  author={Kov{\'a}cs, S{\'a}ndor J.},
  title={\href{http://dx.doi.org/10.1007/BF01450509}%
    {The cone of curves of a $K3$ surface}},
  journal={Math. Ann.},
  volume={300},
  date={1994},
  number={4},
  pages={681--691},
}

\bib{Ina}{article}{
  label={Ina},
  author={Inamdar, S. P.},
  title={\href{http://dx.doi.org/10.2140/pjm.1997.177.71}%
      {On syzygies of projective varieties}},
  journal={Pacific J. Math.},
  volume={177},
  date={1997},
  number={1},
  pages={71--76},
}

\bib{PAGI}{book}{
  label={PAG1},
  author={Lazarsfeld, Robert},
  title={Positivity in algebraic geometry. I},
  series={Ergebnisse der Mathematik und ihrer Grenzgebiete. 3. Folge.},
  volume={48},
  publisher={Springer-Verlag},
  place={Berlin},
  date={2004},
  pages={xviii+387},
}

\bib{PAGII}{book}{
  label={PAG2},
  author={Lazarsfeld, Robert},
  title={Positivity in algebraic geometry. II},
  series={Ergebnisse der Mathematik und ihrer Grenzgebiete. 3. Folge.},
  volume={49},
  publisher={Springer-Verlag},
  place={Berlin},
  date={2004},
  pages={xviii+385},
}

\bib{Loz}{article}{
  label={Loz},
  author={Lozovanu, Victor},
  title={\href{http://dx.doi.org/10.1016/j.jalgebra.2009.07.013}%
    {Regularity of smooth curves in biprojective spaces}},
  journal={J. Algebra},
  volume={322},
  date={2009},
  number={7},
  pages={2355--2365},
}

\bib{MS}{article}{
  author={Maclagan, Diane},
  author={Smith, Gregory G.},
  title={\href{http://dx.doi.org/10.1515/crll.2004.040}%
    {Multigraded Castelnuovo-Mumford regularity}},
  journal={J. Reine Angew. Math.},
  volume={571},
  date={2004},
  pages={179--212},
}

\bib{Mori}{article}{
  label={Mor},
  author={Mori, Shigefumi},
  title={\href{http://projecteuclid.org/euclid.nmj/1118787649}%
    {On degrees and genera of curves on smooth quartic surfaces in ${\bf P}^3$}},
  journal={Nagoya Math. J.},
  volume={96},
  date={1984},
  pages={127--132},
}

\bib{Tev}{article}{
  label={Tev},
  author={Tevelev, Jenia},
  title={\href{http://dx.doi.org/10.1353/ajm.2007.0029}%
    {Compactifications of subvarieties of tori}},
  journal={Amer. J. Math.},
  volume={129},
  date={2007},
  number={4},
  pages={1087--1104},
}

\bib{Wahl}{article}{
  label={Wah},
  author={Wahl, Jonathan},
  title={On cohomology of the square of an ideal sheaf},
  journal={J. Algebraic Geom.},
  volume={6},
  date={1997},
  number={3},
  pages={481--511},
}

		
\end{biblist}
\end{bibdiv}
\end{document}